\documentclass{amsart}

\usepackage{graphicx,algorithm,algpseudocode,xcolor,stmaryrd,amsaddr,tikz,hyperref,amsmath}

\algnewcommand\algorithmicinput{\textbf{Input:}}
\algnewcommand\Input{\item[\algorithmicinput]}
\algnewcommand\algorithmicoutput{\textbf{Output:}}
\algnewcommand\Output{\item[\algorithmicoutput]}
\algnewcommand{\LineIf}[2]{\State \algorithmicif\, #1 \,\algorithmicthen\, #2 \,\algorithmicend\ \algorithmicif}
\algnewcommand{\LineForAll}[2]{\State \algorithmicforall\, #1 \,\algorithmicdo\, #2 \,\algorithmicend\ \algorithmicfor}
\algnewcommand{\Accept}{\textbf{accept}}
\algnewcommand{\Reject}{\textbf{reject}}

\newcommand{\dom}{\mathrm{dom}}

\newcommand{\NL}{\ensuremath\mathsf{NL}}

\newcommand{\PSPACE}{\ensuremath\mathsf{PSPACE}}

\newcommand{\EXPTIME}{\ensuremath\mathsf{EXPTIME}}

\newcommand{\AC}{\ensuremath\mathsf{AC}}
\newcommand{\LOGSPACE}{\ensuremath\mathsf{L}}
\newcommand{\Poly}{\ensuremath\mathsf{P}}
\newcommand{\gR}{ \mathrel{\mathcal{R}}}
\newcommand{\gL}{ \mathrel{\mathcal{L}}}
\newcommand{\gH}{ \mathrel{\mathcal{H}}}
\newcommand{\gI}{ \mathrel{\mathcal{I}}}
\newcommand{\gJ}{ \mathrel{\mathcal{J}}}
\DeclareMathOperator{\Stab}{Stab}

\newtheorem{theorem}{Theorem}[section]
\newtheorem{lemma}[theorem]{Lemma}
\newtheorem{corollary}[theorem]{Corollary}
\newtheorem{proposition}[theorem]{Proposition}

\newtheorem{example}[theorem]{Example}

\theoremstyle{remark}

\numberwithin{equation}{section}
\begin{document}

\title{On the Complexity of Inverse Semigroup Conjugacy}
\date{\today}
\author{Trevor Jack}
\address{NOVAMATH -- Center for Mathematics and Applications \\
Faculdade de Ci\^{e}ncias e Tecnologia \\
Universidade Nova de Lisboa \\
2829-516 Caparica, Portugal}
\email{trevor.jack@colorado.edu}

\thanks{This work was partially supported by the Funda\c{c}\~{a}o para a Ci\^{e}ncia e a Tecnologia (Portuguese Foundation for Science and Technology) through the projects UIDB/00297/2020 (Centro de Matemática e Aplica\c{c}\~{o}es) and PTDC/MAT-PUR/31174/2017.}

\begin{abstract}
We investigate the computational complexity of various decision problems related to conjugacy in finite inverse semigroups. We describe \linebreak polynomial-time algorithms for checking if two elements in such a semigroup are $\sim_p$ conjugate and whether an inverse monoid is factorizable. We describe a connection between checking $\sim_i$ conjugacy and checking membership in inverse semigroups. We prove that $\sim_o$ and $\sim_c$ are partition covering for any countable set and that $\sim_p$, $\sim_{p^*}$, and $\sim_{tr}$ are partition covering for any finite set. We prove that checking for nilpotency, $\gR$-triviality, and central idempotents in partial bijection semigroups are $\NL$-complete problems and we extend several complexity results for partial bijection semigroups to inverse semigroups.
\end{abstract}
\maketitle
\section{Introduction}
Conjugacy classes are important tools for analyzing group structure and significant work has been done to define an equivalence relation for semigroups that coincides with group conjugacy when the semigroup is a group. It turns out there are manys ways to define such an equivalence relation. In \cite{AK:FN}, the authors analyze four ways of defining conjugacy for semigroups and this paper will begin with a brief review of these types of conjugacy. This paper will then build upon results from \cite{AK:CI} related to conjugacy in inverse semigroups, with a particular focus on determining the computational complexity of decision problems involving conjugacy in finite inverse semigroups.

We will also recite several results from \cite{TJ:PB} related to decision problems for partial bijection semigroups and then extend those results to decision problems for inverse semigroups. Relatedly, this paper will give several hardness results left open in \cite{TJ:PB}. Finally, this paper will answer Open Problem 6.10 from \cite{AK:FN}.

\section{Preliminaries} \label{NotationSection}
An element $s \in S$ is \emph{idempotent} if $ss=s$ and we denote the set of idempotents in $S$ as $E(S)$. Every finite semigroup $S$ has an \emph{idempotent power} $\omega \in \mathbb{N}$ such that $s^{\omega} \in E(S)$ for every $s \in S$. If $S$ is a monoid, define $S^1 := S$. Otherwise, let $1$ be a new element satisfying $1s = s1 = s$ for each $s \in S$ and define $S^1 := S \cup \{1\}$. We call an element $u \in S^1$ a \emph{unit} iff there exists $u^{-1} \in S$ satisfying $u u^{-1} = u^{-1} u = 1$. The set of units of $S$ forms a subgroup, denoted $U(S)$. A monoid $S$ is called \emph{factorizable} if $S = U(S)E(S)$. We shall also refer to an element as factorizable if it is the product of a unit and an idempotent.

For a semigroup $S$, two elements $a,b \in S$ are said to be $\gR$-related  ($\gL$-related), denoted $a\gR b$ ($a \gL b$), iff $aS^1 = bS^1$ ($S^1a = S^1b$). The elements are $\gH$-related, denoted $a\gH b$ iff $a\gR b$ and $a\gL b$. For $\gI \in \{\gR,\gL,\gH\}$, a semigroup $S$ is \emph{$\gI$-trivial} iff for every $s,t \in S$, $s \gI t$ implies $s=t$.

A semigroup element $0 \in S$ is a \emph{left (right) zero} if $0a = 0$ ($a0 = 0$) for each $a \in S$. If a semigroup has a left zero and a right zero, then they are equal and we call this the zero element of the semigroup. A semigroup with a zero element is \emph{$n$-nilpotent} if the composition of any $n$ elements of the semigroup yields the zero element. A semigroup is \emph{nilpotent} if it is $n$-nilpotent for some $n \in \mathbb{N}$.

The \emph{full transformation semigroup} over $[n] := \{1,\dots,n\}$, denoted $T_n$, is the set of all mappings $f \colon [n] \to [n]$, together with function composition. Subsemigroups of the full transformation semigroup are often also referred to as \emph{transformation semigroups}. The \emph{full partial bijection semigroup} over $[n]$, denoted $I_n$, is the set of all bijections defined on subsets of $[n]$, together with function composition. Subsemigroups $S \leq I_n$ are called \emph{partial bijection semigroups}. Let $\dom(a)$ denote the domain of $a \in I_n$. Let $[n]a$ denote the range of $a$.

For $s \in S$, we define $t \in S$ to be an \emph{inverse} of $s$ iff $sts = s$ and $tst = t$. A semigroup is an \emph{inverse semigroup} if each of its elements has a unique invere. Every finite inverse semigroup can be embedded into some $I_n$ \cite[Th 5.1.7]{HO:FS}. Every element $s \in I_n$ can be characterized by the orbits of its action on $[n]$, which can be subdivided into two types.

A \emph{cycle of length k} is an ordered subset of $[n]$ denoted $(x_1,\dots,x_k)$ such that $x_is = x_{i+1}$ for $1 \leq i < k$ and $x_ks = x_1$. Let $\Delta_s^k$ be the set of all cycles of length $k$, $\Delta_s$ be the set of all cycles, and $\Delta_s^{>1} := \Delta \setminus \Delta_s^1$. A \emph{chain of length k} is an ordered set denoted $[x_1,\dots,x_k]$ such that $x_is = x_{i+1}$ for $1 \leq i < k$ and $x_k \not \in \dom(s)$. Let $\Theta_s^k$ be the set of all chains of length $k$, $\Theta_s$ be the set of all chains, and $\Theta_s^{>1} := \Theta \setminus \Theta_s^1$. Note that, by this convention, $\Theta_s^1 := \{\{x\}:x \not \in \dom(s)\}$

We now present several approaches for generalizing group conjugacy to semigroups. Each of the following equivalence relations coincide with group conjugacy when the semigroup is a group.

If $S$ is a monoid, we can define conjugacy in terms of units.
$$a \sim_u b \Leftrightarrow uau^{-1} = b \text{ and } u^{-1}bu = a \text{ for some }u \in U(S)$$

We can define conjugacy in more general contexts if we avoid referring to the unary inverse operator. Note that two group elements $a,b \in G$ are conjugate iff there exists $g \in G$ such that $ag = gb$. This suggests the following conjugacy relation over elements $a,b \in S$.
$$a \sim_o b \Leftrightarrow ag = gb \text{ and } bh = ha \text{ for some }g,h \in S$$

This is an equivalence relation, but it suffers the unfortunate drawback of being the universal equivalence relation for any semigroup that has a \emph{zero element}, $0$, satisfying $0s = s0 = 0$ for every $s \in S$.

To avoid this drawback, the authors in \cite{AK:CS} restricted the set from which $g$ and $h$ can be chosen. For a semigroup $S$ with zero element $0$ and any nonzero $a \in S \setminus \{0\}$, define $\mathbb{P}(a) := \{g \in S: (ma)g \neq 0 \text{ for any } ma \in S^1a \setminus \{0\}\}$. If $S$ has no zero element, define $\mathbb{P}(a) := S$. Define $\mathbb{P}^1(a) := \mathbb{P}(a) \cup \{1\}$ with $1$ being the identity element of $S^1$. We now define $\sim_c$ in a similar fashion as $\sim_o$:
$$a \sim_c b \Leftrightarrow ag = gb \text{ and } bh = ha \text{ for some }g \in \mathbb{P}^1(a) \text{ and some } h \in \mathbb{P}^1(b)$$

The next notion is motivated by the fact that $a,b \in G$ are conjugate iff there exist $g,h \in G$ such that $a=gh$ and $b=hg$:
$$a \sim_p b \Leftrightarrow a = uv \text{ and } b=vu \text{ for some } u,v \in S$$

In general, this relation is not transitive, so we denote the transitive closure of this relation as $\sim_{p^*}$. For finite monoids, we define the following notion of conjugacy, which is motivated by the representation theory of semigroups \cite{IR:RT}. We say $a \sim_{tr} b$ iff there are mutually inverse $g,h \in S$ such that: $hg = a^\omega$, $gh = b^\omega$, and $ga^{\omega+1}h = b^{\omega+1}$.

For inverse semigroups, we define a more general version of $\sim_u$.
$$a \sim_i b \Leftrightarrow sas^{-1} = b \text{ and } s^{-1}bs = a \text{ for some }s \in S$$

For finite inverse semigroups, it is known that $\sim_c \, = \, \sim_i$ \cite[Cor 2.15]{AK:FN} \cite[Prop 2.11]{AK:CI} and that the remaining conjugacy relations are ordered by inclusion as follows:
$$\sim_u \, \subseteq \, \sim_i \, \subseteq \, \sim_p \, \subseteq \, \sim_{p^*} \, \subseteq \, \sim_{tr} \, \subseteq \, \sim_o \, \text{\cite[Fig 1.1]{AK:FN}},$$

This paper will be investigating the computational complexity of decision problems in which we are given a set of partial bijections $\{a_1,\dots,a_k\}$ and asked to check some property of the generated inverse semigroup. Note that each partial bijection has a unique inverse and, for any composition of partial bijections $s_1 \cdots s_k$, its unique inverse is $s_k^{-1} \cdots s_1^{-1}$. So, the inverse semigroup generated by $a_1,\dots,a_k$ is the partial bijection semigroup generated by $a_1,\dots,a_k,a_1^{-1},\dots,a_k^{-1}$, which we denote as $\langle a_1,\dots,a_k,a_1^{-1},\dots,a_k^{-1}\rangle$.

Our analysis will reference complexity classes from the following hierarchy:

$$ \AC^0 \subseteq \LOGSPACE \subseteq \NL \subseteq \Poly \subseteq \PSPACE \subseteq \EXPTIME.$$

$\AC^0$ is the class of sets decidable by unbounded fan-in Boolean circuits of constant depth. Equivalently, $\AC^0$ is the class of first-order definable properties \cite{IM:DC}. Hence, it includes all decision problems in which we are given generators $a_1,\dots,a_k \in I_n$ and asked to check a property that can be characterized by a first-order formula quantified over the points $[n]$ and the generators $\{a_1,\dots,a_k\}$. Note that we are not allowed to quantify over all generated elements of the semigroup. $\LOGSPACE$ ($\NL$) is the class of sets decidable by a deterministic (nondeterministic) Turing machine using at most logarithmic space. $\Poly$ ($\PSPACE$) consist of sets that are decidable by a deterministic Turing machine in polynomial time (space). We refer the reader to \cite{PA:CC} for further background on computational complexity and to \cite{CP:AT} and \cite{HO:FS} for further background on semigroup theory.

\section{First-Order Definable Properties for Inverse Semigroups}
Since every finite inverse semigroup is isomorphic to a finite partial bijection semigroup, problems for the former cannot be harder than they are for the latter. The following problems are discussed in \cite{TJ:PB}, where relevant definitions for these results can be found. In particular, the complexity class $\AC^0$ corresponds to properties that can be characterized by first-order formulas quantified over points $[n]$ and generators $a_1,\dots,a_k \in I_n$. Combining results from \cite{FJ:CP} and \cite{TJ:PB}, we obtain the following.

\begin{corollary} \label{cor:naive}
Checking if a finite inverse semigroup given by generators in $I_n$:
\begin{enumerate}
\item is commutative is in $\AC^0$ by \cite[Thm 3.2]{FJ:CP},
\item is a semilattice is in $\AC^0$ by  \cite[Thm 3.3]{FJ:CP},
\item is a group is in $\AC^0$ by  \cite[Thm 3.5]{FJ:CP},
\item has a left identity is in $\AC^0$ by \cite[Thm 3.2]{TJ:PB},
\item has a right identity is in $\AC^0$ by \cite[Thm 3.3]{TJ:PB},
\item is completely regular is in $\AC^0$ by \cite[Thm 3.5]{TJ:PB}, and
\item is Clifford is in $\AC^0$ by \cite[Cor 3.6]{TJ:PB}.
\end{enumerate}
\end{corollary}

We now consider four problems that are known to be $\NL$-complete for transformation semigroups, but are much simpler for inverse semigroups. The input for each of the problems are generators for a semigroup. The respective outputs for the four problems are whether the generated semigroup: (1) is nilpotent, (2) is $\gR$-trivial, (3) has central idempotents, and (4) has a zero element.

That the first three problems are $\NL$-complete for transformation semigroups is essentially a consequence of \cite[Lem 4.9]{FJ:CP} and \cite[Lem 4.12]{FJ:CP}. We now adapt the proofs of those lemmas to prove that the first three problems remain $\NL$-complete even for partial bijection semigroups. We shall be reducing from the $\NL$-complete problem of checking if a directed graph is acyclic.

\begin{lemma}\label{lem:acyclicRed}
There exists a logspace transducer which, given a directed graph $G = (V, E)$, produces a partial bijection semigroup $S \leq I_V$ with the following properties:
\begin{enumerate}
  \item If the graph is acyclic, then every square in $S$ is the zero element.
  \item If the graph is not acyclic, then the semigroup is not $\gR$-trivial and there exists an idempotent that is not central.
\end{enumerate}
\end{lemma}
\begin{proof}
For each edge $(u,v) \in E$, define a partial bijection $a_{(u,v)}$ such that $\dom(a_{(u,v)}) = \{u\}$ and $u a_{(u,v)} = v$. If $G$ is acyclic, then for any $s = a_{(u_1,v_1)} \cdots a_{(u_m,v_m)}$, either $u_1 \neq v_m$ or $v_\ell \neq u_{\ell+1}$ for some $\ell \in [m-1]$. In either case, $s^2$ has empty domain. If $G$ contains a cycle $v_0, \dots, v_\ell$ with $(v_{i-1}, v_i) \in E$ for all $i \in [\ell]$ and with $v_0 = v_\ell$, then the element $s = a_{(v_0, v_1)} \cdots a_{(v_{\ell-1}, v_\ell)}$ is the identity on $v_0$ and is undefined on all other points. Therefore, $s$ is idempotent. Note that $a_{(v_0, v_1)} s$ is the zero element but $s a_{(v_0, v_1)}$ is not, which means that $s$ is not central. Moreover, $s \gR sa_{(v_0,v_1)}$ and, since the graph has no self-loops, $s \neq sa_{(v_0,v_1)}$. Thus $S$ is not $\gR$-trivial.
\end{proof}

\begin{theorem}\label{thm:nilRcomIdem}
Given generators for a partial bijection semigroup $S$, each of the following problems is $\NL$-complete:
\begin{enumerate}
\item checking if $S$ is nilpotent;
\item checking if $S$ is $\gR$-trivial; and
\item checking if every idempotent in $S$ is central.
\end{enumerate}
\end{theorem}
\begin{proof}
Each of these problems are in $\NL$ since they are in $\NL$ for transformation semigroups \cite{FJ:CP}. For $\NL$-hardness, we argue that Lemma~\ref{lem:acyclicRed} reduces the problem of directed graph acyclicity to checking any of the three aforementioned properties. If the graph is acyclic, then by Lemma~\ref{lem:acyclicRed}, $s^2 = 0$ for every $s \in S$ and $E(S) = \{0\}$. Since $S$ is finite, $S$ is nilpotent. Since every nilpotent semigroup is $\gR$-trivial, then $S$ is $\gR$-trivial. Since the zero element is central, all of the idempotents in $S$ are central. Conversely, if the graph is not acyclic, then by Lemma~\ref{lem:acyclicRed}, $S$ is not $\gR$-trivial and $S$ has some non-central idempotent $e$. Since the zero element is central, $e^d = e \neq 0$ for every $d \in \mathbb{N}$ and $S$ is not nilpotent.
\end{proof}

Although the first three problems remain $\NL$-complete for partial bijection semigroups, checking zero membership can be solved in deterministic logspace. For a set of partial bijections $A = \{a_1,\dots,a_k\} \subseteq I_n$, define the {\emph transformation graph} $\Gamma(A,[n+1])$ to have vertices $[n+1]$ and undirected edges as follows: (1) $(p,q) \in [n]^2$ if either $pa_i = q$ or $qa_i = p$ for some $i \in [k]$ and (2) $(p,n+1)$ if $p \in [n] \setminus \dom(a_i)$ for some $i \in [k]$.

\begin{lemma} \label{lem:zero}
Let $A := \{a_1,\dots,a_k\} \subseteq I_n$ and let $S$ be the partial bijection semigroup generated by the elements of $A$. Then $S$ has a zero element iff the only connected component of $\Gamma(A,[n+1])$ containing more than one vertex is the component containing $n+1$.
\end{lemma}
\begin{proof}
$\Rightarrow$: Let $0=a_{i_1} \cdots a_{i_m}$ be the zero element of $S$. We claim that the only connected component of $\Gamma(A,[n+1])$ that contains more than one vertex is the connected component containing $n+1$. Pick any edge $(p,q) \in [n]^2$ in the graph of $\Gamma(A,[n+1])$ such that $p \neq q$. Let $a \in \{a_1,\dots,a_k\}$ be the generator such that, WLOG, $pa=q$. If $q \in \dom(0)$, $q0 = pa0 = p0$ contradicting that $0$ is a partial bijection. Therefore, $q \not \in \dom(0)$, which means there exists $\ell \in [m]$ such that $q a_{i_1} \cdots a_{i_{\ell-1}} \not \in \dom(a_{i_\ell})$. That is, $q$ is connected to $n+1$ in $\Gamma(A,[n+1])$.

$\Leftarrow$: Let $\{x_1,\dots,x_m\} \subseteq [n+1]$ be the connected component containing $n+1$ and assume every other connected component has exactly one vertex. We first claim that for every point $x$ in the connected component of $n+1$, there exists $s \in S$ such that $x \not \in \dom(s)$. Let $(x,y_1),(y_1,y_2),\dots,(y_m,n+1)$ be a path from $x$ to $n+1$. Then there exist generators $a_{i_1},\dots,a_{i_{m+1}}$ such that $xa_{i_1}\cdots a_{i_m} = y_m$ and $y_m \not \in \dom(a_{i_{m+1}})$. Thus, $x \not \in \dom(a_{i_1} \cdots a_{i_{m+1}})$.

Let $s \in S$ be such that $|\{x_1,\dots,x_m\} \cap \dom(s)|$ is of minimal size and suppose for the sake of contradiction that some point $x_i \in \{x_1,\dots,x_m\} \cap \dom(s)$. Then because $\{x_1,\dots,x_m\}$ is a connected component, $x_is \in \{x_1,\dots,x_m\}$. By the claim proved above, there exists $s_i \in S$ such that $x_is \not \in \dom(s_i)$. But then $|\{x_1,\dots,x_m\} \cap \dom(s)| > |\{x_1,\dots,x_m\} \cap \dom(ss_i)|$, contradicting the choice of $s$. Thus, $\{x_1,\dots,x_m\} \cap \dom(s) = \emptyset$. We now want to prove that this $s$ is the zero element for $S$. Note that in a partial bijection semigroup, any right zero is also a left zero, so it is enough to show that $s$ is a right zero.

Since $\{x_1,\dots,x_m\}$ is a connected component, $x_ia_j \in \{x_1,\dots,x_m\}$ for each $i \in [m]$ and each $j \in [k]$. Thus, $x_i \not \in \dom(a_j0)$. Now consider any $x \not \in \{x_1,\dots,x_m\}$. Because it is not connected to any other point, it is fixed by every generator, so $xa_i0 = x0$.
\end{proof}

\begin{theorem}
Given generators $a_1,\dots,a_k \in I_n$, there is a deterministic logspace algorithm for checking if the generated partial bijection semigroup has a zero.
\end{theorem}
\begin{proof}
By Lemma~\ref{lem:zero}, we need only check that every connected component of $\Gamma(\{a_1,\dots,a_k\},[n+1])$ contains exactly one vertex except perhaps the component containing $n+1$. We can do this by checking that each vertex $x \in [n]$ and each edge $(a,b)$ satisfies the following expression: if $x=a$ and $x \neq b$, then $x$ is connected to $n+1$. We only need $3\log(n)$ space to store $x$, $a$, and $b$. We only need logarithmic space to run Reingold's algorithm to check if $x$ is connected to $n+1$ \cite{OR:UC}.
\end{proof}

We now consider these problems in the context of inverse semigroups. Since inverse semigroups are regular and the only regular nilpotent semigroup is the trivial semigroup, checking nilpotency for inverse semigroups is trivial. For two elements of an inverse semigroup $a,b \in S$: (1) $a \gR b$ iff $\dom(a) = \dom(b)$ and (2) $a \gL b$ iff $[n]a = [n]b$ \cite[Prop 5.1.2]{HO:FS}. This immediately yields the following results.

\begin{theorem}
Given generators $a_1,\dots,a_k \in I_n$ and $\gI \, \in \{\gR,\gL,\gH\}$, checking whether $a_1 \gI a_2$ in the generated inverse semigroup is in $\AC^0$.
\end{theorem}

Furthermore, every $\gR$-($\gL$-)class in a regular semigroup has an idempotent, so a regular semigroup is $\gR$-($\gL$-)trivial iff each of its elements is idempotent. The idempotents in inverse semigroups commute, so an inverse semigroup is $\gR$-($\gL$-)trivial iff it is a semilattice, which places the problem in $\AC^0$ by Cor~\ref{cor:naive}.

\begin{corollary}
Given generators $a_1,\dots,a_k \in I_n$, checking if the generated inverse semigroup is $\gR$-trivial or $\gL$-trivial are both in $\AC^0$.
\end{corollary}

\begin{lemma} \label{lem:comIDinv}
Let $S$ be the inverse semigroup generated by $a_1,\dots,a_k \in I_n$. Then all of the idempotents of $S$ are central iff $S$ is completely regular.
\end{lemma}
\begin{proof}
$\Rightarrow$: Since $ss^{-1}$ is idempotent and $ss^{-1}s = s$, then $sss^{-1} = s$. Thus, $s$ permutes its image and generates a subgroup.

$\Leftarrow$: The proof of \cite[Thm 3.5]{TJ:PB} demonstrates that a partial bijection semigroup is completely regular iff $\dom(s_1s_2) = \dom(s_1) \cap \dom(s_2)$ for every $s_1,s_2 \in S$. Thus, for any idempotent $e \in S$, $\dom(es) = \dom(se)$. Since $e$ fixes its domain and range, $es = se$.
\end{proof}

\begin{theorem}
Given generators $a_1,\dots,a_k \in I_n$, checking whether every idempotent in the generated inverse semigroup is central is in $\AC^0$.
\end{theorem}
\begin{proof}
This follows immediately from Corollary~\ref{cor:naive} and Lemma~\ref{lem:comIDinv}.
\end{proof}

\section{Decision Problems Involving Conjugacy}
The preceding results have immediate applications to results from \cite{AK:CI} and \cite{AK:FN}.

\begin{corollary}
Given generators $a_1,\dots,a_k \in T_n$, let $S$ be the generated transformation semigroup. Both of the following problems are in $\AC^0$: (a) checking if $\sim_o$ is the identity relation on $S$ and (b) checking if $\sim_{p^*}$ is the identity relation on $S$.
\end{corollary}
\begin{proof}
First recall that the problems of checking if a transformation semigroup is commutative or if it is a group are both in $\AC^0$ by \cite[Thm 3.2]{FJ:CP} and \cite[Thm 3.5]{FJ:CP}. By \cite[Cor 5.8]{AK:FN}, $\sim_o$ is the identity relation iff $S$ is a commutative group. By \cite[Thm 5.4]{AK:FN}, $\sim_p$ is the identity relation iff $S$ is commutative. Of course, $\sim_p$ is the identity relation iff its transitive closure is.
\end{proof}

\begin{corollary}
Given generators $a_1,\dots,a_k \in I_n$, the problem of checking if $\sim_i$ is the identity relation for the generated inverse semigroup $S$ is in $\AC^0$.
\end{corollary}
\begin{proof}
By \cite[Thm 5.9]{AK:CI}, $\sim_i$ is the identity relation iff $S$ is commutative and checking if $S$ is commutative is in $\AC^0$ by \cite[Thm 3.2]{FJ:CP}.
\end{proof}

\begin{corollary}
Given generators $a_1,\dots,a_k \in T_n$, checking if $\sim_{tr}$ is the identity relation on the generated semigroup is in $\NL$. For partial bijection semigroups, this problem is in $\AC^0$.
\end{corollary}
\begin{proof}
By \cite[Thm 5.5]{AK:FN}, $\sim_{tr}$ is the identity relation iff the generated semigroup is commutative and completely regular. The latter condition can be checked in $\NL$ by \cite[Thm 5.6]{FJ:CP}. For partial bijection semigroups, this condition can be checked in $\AC^0$ by \cite[Thm 3.5]{TJ:PB}.
\end{proof}

We know that for any two elements $a,b \in I_n$, $a \sim_i b$ in $I_n$ iff $a$ and $b$ have the same cycle-chain type \cite[Thm 2.10]{AK:CI}. Similarly, $a \sim_{p^*} b$ iff $a \sim_{tr} b$ iff $a$ and $b$ have the same cycle type \cite[Thm 10]{KM:OC}. The authors in \cite[Lemma 14]{KM:OC} prove that $a \sim_p b$ iff (1) $|\Delta_a^k| = |\Delta_b^k|$ for each $k \in [n]$ and (2) there exists a partial bijection between their chains satisfying several conditions. We now prove a modified version of those conditions to obtain a polynomial-time algorithm to check if $a \sim_p b$.

For any $a,b \in I_n$, let $\phi: \Theta_a^{>1} \hookrightarrow \Theta_b$ and $\psi: \Theta_b^{>1} \hookrightarrow \Theta_a$. We call $(\phi,\psi)$ a \emph{satisfying pair} if $|\theta| \leq |\phi(\theta)|+1$ and $|\theta'| \leq |\psi(\theta')|+1$ for every $\theta \in \Theta_a^{>1}$ and every $\theta' \in \Theta^{>1}$. The pair is an \emph{inversive satisfying pair} if, additionally, $\phi \circ \psi$ and $\psi \circ \phi$ are both idempotent.

\begin{lemma} \label{lem:pConj}
For any $a,b \in I_n$, the following are equivalent:
\begin{enumerate}
\item $a \sim_p b$ in $I_n$;
\item $|\Delta_a^k| = |\Delta_b^k|$ for each $k \in [n]$ and there exists a satisfying pair; and
\item $|\Delta_a^k| = |\Delta_b^k|$ for each $k \in [n]$ and there exists an inversive satisfying pair.
\end{enumerate}
\end{lemma}
\begin{proof}
$(1) \Rightarrow (2)$: Let $a = uv$ and $b = vu$ for some $u,v \in I_n$. Pick any $1 \leq k \leq n$ and any $(x_1 \, \dots \, x_k) \in \Delta_a^k$. Then $x_k uvu = x_1 u$ and, for each $i < k$, $x_i uvu = x_{i+1}u$. That is, $(x_1 u \, \dots \, x_k u) \in \Delta_b^k$. Because $u$ is a partial bijection, we know that for any other $(y_1 \, \dots \, y_k) \in \Delta_a^k$, $(x_1 u \, \dots \, x_n u)$ and $(y_1  u \, \dots \, y_k u)$ are distinct cycles of $b$. Thus, $u$ can be considered an injective map from $\Delta_a^k$ to $\Delta_b^k$. Symmetrically, $v$ is an injective map from $\Delta_b^k$ to $\Delta_a^k$, so $|\Delta_a^k|=|\Delta_b^k|$.

Pick any $1 < k \leq n$ and any $[x_1 \, \dots \, x_k] \in \Theta_a^k$. Then for each $1 \leq i < k-1$, $x_i uvu = x_{i+1} u$. Because $x_k \not \in \dom(a)$, then $x_k \not \in \dom(u)$ or $x_ku \not \in \dom(vu)$. Thus, $u$ induces a map $\phi: \Theta_a^{>1} \rightarrow \Theta_b$ such that for each $\theta \in \dom(\phi)$, $|\theta| \leq |\phi(\theta)|+1$. Symmetrically, $v$ induces a map $\psi: \Theta_b^{>1} \rightarrow \Theta_a$ such that for each $\theta \in \dom(\psi)$, $|\theta| \leq |\psi(\theta)|+1$. That these maps are injective follows from $u$ and $v$ being partial bijections.

$(2) \Rightarrow (3)$: Let $\phi$ and $\psi$ be a satisfying pair. We prove by induction on the size of $|\Theta_a^{>1}|+|\Theta_b^{>1}|$ that the existence of a satisfying pair implies existence of an inversive satisfying pair. If either $\Theta_a^{>1}$ or $\Theta_b^{>1}$ is empty, then any satisfying pair will already be an inversive satisfying pair since the compositions will have empty domain. Suppose an inversive pair $(\phi,\psi)$ exists for $a$ and $b$. Pick a maximal length chain $\theta_a \in \Theta_a^{>1}$ and a maximal length chain $\theta_b \in \Theta_b^{>1}$. We want $\phi(\theta_a) = \theta_b$ and $\psi(\theta_b) = \theta_a$. If this is not already true, we will redefine $\phi$ and $\psi$ in such a way that $(\phi,\psi)$ remains a satisfying pair.

Let $\eta_a := \phi(\theta_a)$ and $\eta_b := \psi(\theta_b)$. If there exists $\theta_a' \in \Theta_a^{>1}$ such that $\phi(\theta_a') = \theta_b$, define $\phi(\theta_a) := \theta_b$ and $\phi(\theta_a') := \eta_a$. Since $\theta_a$ is of maximal length, $|\theta_a'| \leq |\theta_a|$. Since $(\phi,\psi)$ is a satisfying pair, $|\theta_a| \leq |\eta_a|+1$. Since $\theta_b$ is of maximal length, $|\theta_a| \leq |\eta_a|+1 \leq |\theta_b|+1$. If there is no $\theta_a' \in \Theta_a^{>1}$ such that $\phi(\theta_a') = \theta_b$, simply define $\phi(\theta_a) := \theta_b$. Likewise, if there exists $\theta_b' \in \Theta_b^{>1}$ such that $\psi(\theta_b') = \theta_a$, then swap the images of $\psi(\theta_b)$ and $\psi(\theta_b')$. If no such $\theta_b'$ exists, simply define $\psi(\theta_b) := \theta_a$. We now have a satisfying pair $(\phi,\psi)$ such that $\phi(\theta_a) = \theta_b$ and $\psi(\theta_b) = \theta_a$.

Let $a'$ be equal to $a$ restricted to $[n] \setminus \theta_a$ and $b'$ be equal to $b$ restricted to $[n] \setminus \theta_b$. Restricting $\phi$ to $\Theta_a^{>1} \setminus \{\theta_a\}$ and $\psi$ to $\Theta_b^{>1} \setminus \{\theta_b\}$ yields a satisfying pair for $a'$ and $b'$. By the induction hypothesis, there exists an inversive satisfying pair $(\phi',\psi')$ for $a'$ and $b'$. Extend these maps by defining $\phi'(\theta_a) := \theta_b$ and $\psi(\theta_b) := \theta_a$. Now $(\phi',\psi')$ is an inversive satisfying pair for $a$ and $b$.

$(3) \Rightarrow (1)$: We will use the conditions on the cycles and the inversive satisfying pair $(\phi,\psi)$ to construct $u$ and $v$ such that $a = uv$ and $b = vu$. For any $1 \leq k \leq n$, the cycle condition ensures that we can define a bijection $f:\Delta_a^k \rightarrow \Delta_b^k$. For each $(x_1 \dots x_k) \in \Delta_a^k$, define $u$ such that $f((x_1 \dots x_k)) = (x_1 u \dots x_k u)$. Although there are $k$ ways to define $u$ in this way, any of them will suffice. For each of these $k$ options, we can define $v$ such that $x_k uv = x_1$ and $x_i uv = x_{i+1}$ for each $1 \leq i < k$. Then $(x_1 \dots x_k) \in \Delta_{uv}^k$ and $f((x_1 \dots x_k)) \in \Delta_{vu}^k$. That is, $\Delta_a = \Delta_{uv}$ and $\Delta_b = \Delta_{vu}$.

We will use the inversive satisfying pair $(\phi,\psi)$ to continue constructing $u$ and $v$ by the following algorithm.

Step 1. Define $C$ to be the chains in the disjoint union of $\Theta_a^{>1}$ and $\Theta_b^{>1}$.

Step 2. Pick a maximal length chain $\theta \in C$.

Step 3(u). Suppose $\theta = [x_1 \dots x_k] \in \Theta_a^{>1}$. If $|\phi(\theta)| = |\theta|$, define $u$ so that $\phi(\theta) = [x_1u \dots x_ku]$. If $|\phi(\theta)| = |\theta| - 1$, let $x_ku$ be undefined. Define $v$ to satisfy $x_iuv = x_{i+1}$ for each $1 \leq i < k$.

Step 3(v). If $\theta \in \Theta_b^{>1}$, execute Step 3(u) with $u$ and $v$ transposed and replacing $\phi$ with $\psi$.

Step 4. Remove $\theta$ from $C$. If Step 3(u) (3(v), resp.) was just executed, let $C:= C \setminus \{\phi(\theta)\}$ ($\{\psi(\theta)\}$, resp.). If $C$ is not empty, return to Step 2.

We first claim that the two cases considered in Step 3(u) are the only possible cases. Certainly, the claim is true if $|\phi(\theta)| = 1$, so we need only consider when $\phi(\theta) \in C$ at the start of the algorithm. Because $(\phi,\psi)$ is a satisfying pair, we know $|\theta| \leq |\phi(\theta)|+1$. If we prove that $\theta$ is picked in Step 3(u) only if $\phi(\theta)$ is removed from $C$ in Step 4 of that same iteration of the algorithm, then the maximality of $\theta$ finishes the proof of our claim.

Suppose $\phi(\theta)$ is removed from $C$ in Step 4 of the $m^{th}$ iteration of the algorithm. If Step 3(u) was invoked in the $m^{th}$ iteration, then the injectivity of $\phi$ ensures that $\theta$ was picked during this iteration. If Step 3(v) was invoked in the $m^{th}$ iteration, then $(\phi,\psi)$ being an inversive satisfying pair ensures that $\psi(\phi(\theta)) = \theta$ will also be removed so that $\theta$ will not be picked in a subsequent iteration of Step 3(u).

We now verify that $\Theta_{uv}^k = \Theta_a^k$ and $\Theta_{vu}^k = \Theta_b^k$ for each $1 < k \leq n$, which completes the proof. The only way $uv$ is defined on a chain $\theta \in \Theta_{uv}^k$ is if Step 2 picked $\theta$ from $\Theta_a^k$ or $\psi^{-1}(\theta)$ from $\Theta_b^{>1}$. In either case, $\theta \in \Theta_a^k$. For converse containment, pick any $\theta = [x_1,\dots,x_k] \in \Theta_a^k$ and note that, in order to be removed from $C$ in Step 4, either $\theta$ is picked in Step 2 or $\theta = \psi(\theta')$ for some $\theta' \in \Theta_b^{>1}$ picked in Step 2. If $\theta$ is picked, Step 3(u) defines $x_iuv = x_{i+1}$ for each $1 \leq i < k$ and leaves $x_k uv$ undefined. Thus, $\theta \in \Theta_{uv}^k$.

If $\theta$ is not picked, there exists a $\theta' = [y_1,\dots,y_\ell] \in \Theta_b^{>1}$ such that $\theta = \psi(\theta')$ and $\ell = k$ or $\ell = k+1$. Step 3(v) defines $y_iv = x_i$ and $y_ivu = y_{i+1}$ for $i \in [\ell-1]$. Consequently, $x_i uv = y_i vuv = y_{i+1} v = x_{i+1}$ for each $i \in [\ell-1]$. If $\ell = k$, $y_kv=x_k$, but $x_ku$ is left undefined. If $\ell = k+1$, then Step 3(v) omits $y_{k+1}$ from the domain of $v$. Thus, $y_kvu \not \in \dom(v)$ and $x_k \not \in \dom(uv)$. Hence, $\theta \in \Theta_{uv}^k$ and $\Theta_{uv}^k = \Theta_a^k$. By analogous argument, $\Theta_{vu}^k = \Theta_b^k$.
\end{proof}

\begin{theorem}
Given $a,b \in I_n$, there exist polynomial time algorithms to check each of the following in $I_n$: $a \sim_i b$, $a \sim_p b$, $a \sim_{p^*} b$, and $a \sim_{tr} b$.
\end{theorem}
\begin{proof}
For $\sim_c$, $\sim_{p^*}$, and $\sim_{tr}$, we need only enumerate the cycles and chains. Certainly this can be done in linear time simply by tracking how $a$ and $b$ act on $[n]$. To check $\sim_p$, again we check that the cycles match and then we run Algorithm~\ref{alg:pConj} to check whether there exists an inversive satisfying pair.

\emph{Correctness:} If Algorithm~\ref{alg:pConj} accepts, the pairings it chooses for the chains in $\Theta_a$ and $\Theta_b$ will yield an inversive satisfying pair. For each iteration of lines 10-16 (17-22, resp.), define $\phi(\theta) = \theta'$ ($\psi(\theta) = \theta'$, resp.). If $|\theta'|>1$, define $\psi(\theta') = \theta$ ($\phi(\theta') = \theta$, resp.). Lines 10-12 and 17-19 ensure that $\phi$ and $\psi$ have correct domain and range. Lines 13-16 and 20-23 ensure that $(\phi,\psi)$ is an inversive satisfying pair.

We prove the other direction by induction on the size of $C = \Theta_a^{>1} \sqcup \Theta_b^{>1}$. Certainly, the algorithm accepts if $|C| = 0$. Assume that whenever $|C| \leq n$, the existence of an inversive satisfying pair implies the algorithm accepts. Suppose there exists an inversive satisfying pair $(\phi,\psi)$ for $a,b \in I_n$ with corresponding $|C| = n + 1$. WLOG, let $\theta \in \Theta_a$ and $\theta'$ be the first chains chosen by the algorithm in lines 9 and 13. Since $(\phi,\psi)$ is a satisfying pair, $\phi(\theta) \in (C \cap \Theta_b) \cup B$ and $\theta \leq |\phi(\theta)|+1$. Since the algorithm picks a maximal $\theta'$, then $|\theta| \leq |\phi(\theta)|+1 \leq |\theta'|+1$. The algorithm then removes $\theta$ and $\theta'$ from their respective sets so that $C$ now has chains from elements $a'$ and $b'$ obtained by restricting $a$ to $[n] \setminus \theta$ and $b$ to $[n] \setminus \theta'$. We now show that an inversive satisfying pair exists for $a'$ and $b'$ and thus, by inductive hypothesis, the algorithm accepts.

If $\phi(\theta) = \theta'$, then $(\phi,\psi)$ restricted to $C$ will be an inversive satisfying pair. If $\phi(\theta) \neq \theta'$, we consider two cases. If there is no $\eta \in \Theta_a^{>1}$ such that $\phi(\eta) = \theta'$, then because $(\phi,\psi)$ is an inversive satisfying pair, $|\theta'| = 1$ and $(\phi,\psi)$ restricted to $C$ will be an inversive satisfying pair. If there exists $\eta \in \Theta_a^{>1}$ such that $\phi(\eta) = \theta'$, then redefine  $\phi(\eta) := \phi(\theta)$ and $\psi(\phi(\theta)) := \eta$. Since $\theta$ is maximal and $(\phi,\psi)$ was a satisfying pair, $|\eta| \leq |\theta| \leq |\phi(\theta)|+1 = |\phi(\eta)|+1$. Likewise, since $\theta'$ is also maximal, $|\phi(\theta)| \leq |\theta'| \leq |\psi(\theta')|+1 = |\eta|+1$. So, $(\phi,\psi)$ restricted to $C$ is now an inversive satisfying pair.

\emph{Runtime:} As previously noted, Line 1 will execute in linear time. The size of $C$ is at most $n$. Each complete iteration of the while loop removes either one or two chains from $C$. Each iteration picks maximal elements in linear time. Thus, the algorithm accepts or rejects in polynomial time.
\end{proof}

  \begin{algorithm}
  \caption{$\Poly$ algorithm to check if $a \sim_p b$ in $I_n$}
    \label{alg:pConj}
    \begin{algorithmic}[1]
      \Input{$a,b \in I_n$}
      \Output{Is $a \sim_p b$ in $I_n$?}
      \State Enumerate the cycles of $a$ and $b$: $\Delta_a$ and $\Delta_b$
      \For{$k \in [n]$}
        \If{$|\Delta_a^k| \neq |\Delta_b^k|$}
          \Reject
        \EndIf
      \EndFor
      \State Enumerate the chains of $a$ and $b$: $\Theta_a$ and $\Theta_b$
      \State Let $C$ be the disjoint union of $\Theta_a^{>1}$ and $\Theta_b^{>1}$
      \While{$C$ is nonempty}
        \State Pick a maximal length $\theta \in C$
        \If{$\theta \in \Theta_a$}
          \If{$(C \cap  \Theta_b) \cup \Theta_b^1$ is empty}
            \Reject
          \EndIf
          \State Pick a maximal length $\theta '\in (C \cap  \Theta_b) \cup \Theta_b^1$
          \If{$|\theta|>|\theta'|+1$}
            \Reject
          \EndIf
          \State Remove $\theta$ from $C$ and $\theta'$ from $C \cup \Theta_b^1$
        \Else
          \If{$(C \cap  \Theta_a) \cup \Theta_a^1$ is empty}
            \Reject
          \EndIf
          \State Pick a maximal length $\theta '\in (C \cap  \Theta_a) \cup \Theta_a^1$
          \If{$|\theta|>|\theta'|+1$}
            \Reject
          \EndIf
          \State Remove $\theta$ from $C$ and $\theta'$ from $C \cup \Theta_a^1$
        \EndIf
      \EndWhile
      \State\Accept
    \end{algorithmic}
  \end{algorithm}

For an inverse monoid $S$, \cite[Cor 4.2]{AK:CI} proves that ${\sim_u} = {\sim_i}$ iff $S$ is factorizable. Thus, we are interested in the complexity of the following problem.

\medskip
{\bf Factorizable}
\begin{itemize}
\item Input: $a_1,\dots,a_k \in I_n$.
\item Problem: Is $\langle a_1,\dots,a_k,a_1^{-1},\dots,a_k^{-1},1\rangle$ factorizable?
\end{itemize}

\begin{theorem}
Factorizable is in $\Poly$.
\end{theorem}
\begin{proof}
Let $S$ be an inverse monoid and suppose $s,t \in S$ factor as $s = ue$ and $t = vf$ for units $u$ and $v$ and idempotents $e$ and $f$. Then $st = uevf = uevv^{-1}vf$. Since idempotents commute, $st = uvv^{-1}evf$. Since $v^{-1}ev$ is idempotent and the composition of idempotents is an idempotent, $st$ is also factorizable. Thus, an inverse semigroup given by generators is factorizable iff each generator is factorizable. Consequently, it is sufficient to describe a polynomial-time algorithm to check if a generator is factorizable.

Suppose that $a = ue$ for a unit $u$ and an idempotent $e$. Then $xa = xue$ for each $x \in \dom(a)$ and $xu \in \dom(e)$. Since idempotents fix their domain, then $xa = xu$ for each $x \in \dom(a)$. Conversely, if $xu = xa$ for every $x \in \dom(a)$, then $xua^{-1}a = xaa^{-1}a = xa$. Thus, we can further restrict our attention on whether we can produce a unit $u$ that agrees with $a$ on its domain.

Because $U(S)$ is the topmost $\gJ$-class, $S \setminus U(S)$ is an ideal and, consequently, each unit of $S$ is generated by $U(S) \cap \{a_1,\dots,a_k\}$. We can adapt Sims' stabilizer chain algorithm for testing membership in groups. Because we only care how $u$ acts on $\dom(a)$, we use $\dom(a)$ as our base set and employ Sim's algorithm to produce a strong generating set (SGS) for the following stabilizer chain. Let $A_0 = \emptyset$, $A_{i+1} = A_i \cup \{x\}$ for some $x \in \dom(a) \setminus A_i$, $S_i = \Stab(A_i) \leq U(S)$, and $S_\ell = \Stab(\dom(a))$ for $\ell = |\dom(a)|$. A unit $u$ that agrees with $a$ on its domain will exist iff sifting $a$ yields coset representatives $r_i \in S_i$ for each $S_{i+1}$ such that $xa = xr_1 \cdots r_\ell$ for each $x \in \dom(a)$. Since each representative can be generated from the SGS in polynomial time and since the SGS can be enumerated in polynomial time, we can check if $a$ is factorizable in polynomial time.
\end{proof}

We now consider the following problems.

\medskip
{\bf Membership in Inverse Semigroups}
\begin{itemize}
\item Input: $a_1,\dots,a_k,b \in I_n$.
\item Problem: Is $b \in \langle a_1,\dots,a_k,a_1^{-1},\dots,a_k^{-1}\rangle$?
\end{itemize}

\medskip
{\bf Idempotent Membership in Inverse Semigroups}
\begin{itemize}
\item Input: $a_1,\dots,a_k,b \in I_n$ with $b^2 = b$.
\item Problem: Is $b \in \langle a_1,\dots,a_k,a_1^{-1},\dots,a_k^{-1}\rangle$?
\end{itemize}

\medskip
{\bf Subsemigroup $i$-Conjugacy}
\begin{itemize}
\item Input: $a_1,\dots,a_k \in I_n$.
\item Problem: Is $a_1 \sim_i a_2$ in $\langle a_1,\dots,a_k,a_1^{-1},\dots,a_k^{-1}\rangle$?
\end{itemize}

It is known that Membership in Inverse Semigroups can by reduced in polynomial time to Idempotent Membership in Inverse Semigroups \cite[Thm 4.3]{TJ:PB}. We now demonstrate a connection between all three of these problems.

Let $S = \langle a_1,\dots,a_k\rangle \leq I_n$ be an inverse semigroup and let $I_A$ be the semigroup of partial bijective maps on the set $A := [n] \times [n+1]$. For each $s \in S$, define $\bar{s}$ to be the element in $I_A$ satisfying the following: $(x,y)\bar{s} = (xs,y)$. Pick an $h \in I_n$ and define $a,b \in I_A$ as follows:
\begin{align*}
(x,y)a &= (x,y+1),\text{ if }y \leq x\text{ and }x\in dom(h^{-1})\\
(x,y)b &= (x,y+1),\text{ if }y \leq xh\text{ and }x\in dom(h)
\end{align*}

Define $\bar{S} = \langle \bar{a}_1, \dots, \bar{a}_k,a,b,a^{-1},b^{-1}\rangle$. 

\begin{proposition} \label{thm:iRed}
$h \in S$ iff both $hh^{-1} \in S$ and $a \sim_i b$ in $\bar{S}$.
\end{proposition}
\begin{proof}
$\Rightarrow$: First note that $a$ and $b$ are central, so $a \sim_i b$ in $\bar{S}$ iff there exists $\bar{g} \in \langle \bar{a}_1,\dots,\bar{a}_k\rangle$ such that $\bar{g}b\bar{g}^{-1}=a$ and $\bar{g}^{-1}a\bar{g} = b$. 

Certainly if $h \in S$, then $hh^{-1} \in S$. We claim that $\bar{h}a\bar{h}^{-1} = b$ and $\bar{h}^{-1}b\bar{h} = a$. Pick any $(x,y)$ with $y \leq xh$ and $x \in dom(h)$. Then,
\begin{align*}
(x,y)\bar{h}a\bar{h}^{-1} &= (xh,y)a\bar{h}^{-1}\\
&=(xh,y+1)\bar{h}^{-1}\\
&=(x,y+1)\\
&=(x,y)b
\end{align*}
Pick any $(x,y)$ with $y \leq x$ and $x \in dom(h^{-1})$. Then,
\begin{align*}
(x,y)\bar{h}^{-1}b\bar{h} &= (xh^{-1},y)b\bar{h}\\
&=(xh^{-1},y+1)\bar{h}\\
&=(x,y+1)\\
&=(x,y)a
\end{align*}
$\Leftarrow$: Let $\bar{g} \in \bar{S}$ satisfy $\bar{g}a\bar{g}^{-1} = b$ and $\bar{g}^{-1}b\bar{g}=a$. Pick any $x \in dom(h)$ and note that $(x,xh+1)=(x,xh)b=(x,xh)\bar{g}a\bar{g}^{-1}=(xg,xh)a\bar{g}^{-1}$. Thus, $xh \leq xg$ and $xg \in dom(h^{-1})$ for every $x \in dom(h)$.

We now prove by induction that for every $x \in \dom(h)$, $xg \leq xh$ and thus $xg = xh$. Order $\dom(h^{-1}) = \{x_1h,\dots,x_\ell h\}$ so that $x_ih < x_jh$ for every $i<j$. Since $x_\ell g \in \dom(h^{-1})$, then $x_\ell g \leq x_\ell h$ and thus $x_\ell g = x_\ell h$. Assume $x_i g = x_i h$ for every $i>j$. Then $x_j g \in \{x_1h,\dots,x_jh\}$. Then $x_jg \leq x_jh$ and we have that $xg = xh$ for every $x \in dom(h)$. Consequently, $h = hh^{-1}g \in S$.
\end{proof}

In summary, we know these problems are in $\PSPACE$, but whether this bound is tight is an open question. We know that idempotent membership is a $\PSPACE$-complete problem for partial bijection semigroups \cite[Thm 4.1]{TJ:PB}. The following example indicates that checking idempotent membership in inverse semigroups and $i$-conjugacy may also be hard.

\begin{example}
Towers of Hanoi as an inverse semigroup.
\end{example}
The Towers of Hanoi puzzle starts with a tower of $n$ distinctly sized disks on one of three pegs and the challenge is to move the tower to another peg subject to two rules: only one disk may be moved at a time and each disk can only rest on top of a larger disk. The following encodes this puzzle as an inverse semigroup in such a way that the idempotent whose membership is being tested corresponds to moving the disks from one peg to another. Consequently, the minimum number of compositions to generate the idempotent will be exponential in the number of generators and points.

The set of points $\{1\} \cup \{(x,y): x \in [3], y \in [n]\}$ represents $3n$ disk-positions for $n$ disks and an additional point to ensure the disks are moved. For each $i \in [n]$, the generator $a_i$, ($b_i$, $c_i$) corresponds to moving the $i^{th}$ disk from the first (first, second) peg to the second (third, third) peg. These generators fix the point $1$. The idempotents $d$ ($e$) fix only the points $\{(1,y):y\in[n]\}$ ($\{(2,y):y\in[n]\})$.

\[ (x,y) a_i = \begin{cases} 
      (2,i) & \text{if } x = 1 \text{ and } y = i \\
      (x,y) & \text{if } x = 3 \text{ or } y > i
   \end{cases}
\]
\[ (x,y) b_i = \begin{cases} 
      (3,i) & \text{if } x = 1 \text{ and } y = i \\
      (x,y) & \text{if } x = 2 \text{ or } y > i
   \end{cases}
\]
\[ (x,y) c_i = \begin{cases} 
      (3,i) & \text{if } x = 2 \text{ and } y = i \\
      (x,y) & \text{if } x = 1 \text{ or } y > i
   \end{cases}
\]

Let $S$ be the inverse semigroup generated by the aforementioned generators and $d$. Consider whether $e \in S$. Only $d$ excludes the point $1$, so $e = s_1ds_2$ for some $s_1,s_2 \in S$. Note that $d$ and $e$ have equally sized domains, so $s_1$ maps $\dom(e)$ to $\dom(d)$. The puzzle arises from the domain limitations for the $a,b,c$ elements. Note that each excludes disk-positions for disks of equal or lower size on the source peg and the target peg. To preserve $(2,1) \in \dom(e)$, the first generator for $s_1$ must be either $a_1^{-1}$ (to move top disk from peg 2 to peg 1) or $c_1$ (to move top disk from peg 2 to peg 3) . Likewise, each successive generator corresponds to choices for moving disks. Thus, $e$ can only be generated by properly moving the disks, requiring $2^n-1$ compositions. If we include $e$ in the semigroup, our construction shows that $d \sim_i e$ only by way of an exponentially long composition of $a,b,c$ elements.

\section{Partition Covering} \label{sec:open}
A conjugacy relation $\sim$ is \emph{partition covering} for a set $X$ if, for any partition of $X$, there exists a semigroup with universe $X$ such that the partition equals the semigroup's $\sim$-conjugacy classes. Problem 6.10 from \cite{AK:FN} asks whether $\sim_o$, $\sim_p$, and $\sim_{tr}$ are partition covering. It is known, through brute force calculation using the GAP package \emph{Smallsemi}, that $\sim_o$ and $\sim_p$ are partition covering for any $|X| \leq 6$. We now give an affirmative answer to Problem 6.10 for: (1) $\sim_o$ and $\sim_c$ with respect to countable sets and (2) $\sim_p$, $\sim_p^*$, and $\sim_{tr}$ for all finite sets.

\begin{theorem} \label{th:prob10}
The conjugacies $\sim_o$ and $\sim_c$ are partition covering for any countable set. The conjugacies $\sim_p$, $\sim_{p^*}$, and $\sim_{tr}$ are partition covering for any finite set.
\end{theorem}
\begin{proof}
For any partition of $X$, let $\kappa$ be the number of classes in the partition and denote the classes as $\{X_j\}_{0\leq j \leq \kappa - 1}$. Denote the elements of $X_j$ as $\{(i,j)\}_{i \in |X_j|}$. To define a semigroup whose $o$-conjugacy classes are $\{X_j\}_{0 \leq j \leq \kappa - 1}$, we define the following binary operation on $X$ for finite $\kappa$: $(a,b)(c,d) := (1, b+d \mod\kappa)$. For countably infinite $\kappa$, define $(a,b)(c,d) := (1,b+d)$

This operation is closed since $(1,j) \in X_j$ for every $0 \leq j \leq \kappa - 1$. It is associative since the first component behaves like a null semigroup while the second component inherits associativity from addition. We prove that the partition equals the $\sim_o$-conjugacy classes by proving $(a,b) \sim_o (c,d)$ iff $b=d$. Suppose there exists $(e,f) \in X$ such that $(a,b)(e,f) = (e,f)(c,d)$. Since $f < \kappa$, $b+f \equiv f+d mod(\kappa)$ and $b+f = f+d$ each imply $b=d$. Conversely, for any $(a,b),(c,b) \in X_b$, $(a,b)(1,0) = (1,b) = (1,0)(c,b)$. Likewise, $(1,0)(a,b) = (1,b) = (c,b)(1,0)$. Since the semigroup does not have a zero element, this construction also proves that $\sim_c$ is partition covering for countable $X$.

We now consider any finite $X$. To define a semigroup whose $p$-conjugacy classes equal the partition, WLOG let $X_0$ be of maximal size and define the following binary operation on $X$:
\[(a,b)(c,d) :=
\begin{cases}
(a,b) & b=d \\
(a,0) & b \neq d
\end{cases}\]

Because $X_0$ has maximal size, $(a,0) \in X_0$ for any $(a,b) \in X$. It is associative since the first component behaves like a left zero semigroup and the second component is a semilattice with every non-zero element immediately above zero. That the second component is commutative also ensures $(a,b) \sim_p (c,d) \Rightarrow b = d$. Conversely, for any $(a,b),(c,b) \in X_b$, $(a,b) = (a,b)(c,b)$ and $(c,b) = (c,b)(a,b)$. Thus, the $\sim_p$-conjugacy classes equal the partition.  Since these classes are disjoint, $\sim_p$ is transitive, so $\sim_{p^*}$ is also partition covering.

Finally, we prove that $\sim_p \, = \, \sim_{tr}$ for this semigroup. Note that every element is idempotent, so $x^{\omega+1} = x$ for each $x \in X$. Pick any $x_a,x_b \in X$ such that there exist mutually inverse $x_g,x_h \in X$ satisfying $x_gx_h = x_a$, $x_hx_g = x_b$, and $x_gx_ax_h = x_b$. Since $x_gx_hx_g = x_g$ and $x_hx_gx_h = x_h$, their second components must match. Consequently, the second components of $x_a$ and $x_b$ must match. Conversely, every pair $(a,b),(c,b) \in X_b$ are mutually inverse, $(a,b)(c,b) = (a,b)$, $(c,b)(a,b) = (c,b)$, and $(c,b)(a,b)(a,b) = (c,b)$. Thus, $(a,b) \sim_{tr} (c,b)$.
\end{proof}

\section{Open Problems}
{\bf Problem 6.1}: Are the following problems $\PSPACE$-complete: (1) Membership in Inverse Semigroups, (2) Idempotent Membership in Inverse Semigroups, and (3) Subsemigroup $i$-Conjugacy?

{\bf Problem 6.2} Are $\sim_p$, $\sim_p^{*}$, and $\sim_{tr}$ partition covering for any countable set? For which sets is $\sim_c \cap \sim_p$ partition covering?

\section{Acknowledgements}
The author would like to thank Alan Cain, Ant\'onio Malheiro, and Peter Mayr for their valuable comments and contributions.


\iftrue

\fi
\end{document}